\documentclass[11pt]{article}
\usepackage[a4paper,margin=27mm]{geometry}
\usepackage{amsmath,amssymb,amsthm,mathtools}
\usepackage{enumitem}
\usepackage{needspace}
\usepackage{microtype}
\usepackage{hyperref}
\hypersetup{
  colorlinks=true,
  linkcolor=blue,
  citecolor=blue,
  urlcolor=blue,
  pdftitle={Type Annihilation for Classifying Maps of Rack Spaces},
  pdfauthor={Takefumi Nosaka},
  pdfkeywords={rack, quandle, rack space, classifying map, associated group, rack homology, group homology, type of a rack}
}

\newtheorem{theorem}{Theorem}[section]
\newtheorem{proposition}[theorem]{Proposition}
\newtheorem{lemma}[theorem]{Lemma}
\newtheorem{corollary}[theorem]{Corollary}
\newtheorem{remark}[theorem]{Remark}

\DeclareMathOperator{\As}{As}
\DeclareMathOperator{\Inn}{Inn}
\DeclareMathOperator{\Stab}{Stab}
\DeclareMathOperator{\Tor}{Tor}
\DeclareMathOperator{\sgn}{sgn}
\DeclareMathOperator{\Type}{Type}
\DeclareMathOperator{\rank}{rank}
\DeclareMathOperator{\Ker}{Ker}
\DeclareMathOperator{\Coker}{Coker}
\newcommand{\ZZ}{\mathbb Z}
\newcommand{\QQ}{\mathbb Q}
\newcommand{\OO}{\mathcal O}
\newcommand{\tB}{\widetilde B}
\newcommand{\tkap}{\widetilde\kappa}
\newcommand{\tF}{\widetilde F}
\newcommand{\tE}{\widetilde E}
\newcommand{\tP}{\widetilde P}
\newcommand{\td}{\widetilde\partial}
\newcommand{\gr}{\mathrm{gr}}
\newcommand{\R}{\mathrm R}
\newcommand{\Q}{\mathrm Q}
\newcommand{\D}{\mathrm D}
\newcommand{\ab}{\mathrm{ab}}

\title{Type Annihilation for Classifying Maps of Rack Spaces}
\author{Takefumi Nosaka}
\date{}

\begin{document}
\maketitle

\begin{abstract}
We study the classifying map $c\colon BX\to K(\As(X),1)$ of a rack $X$ of finite type. Let $t = \Type(X)$. We prove that $t c_{n*}=0$ for every $n\ge2$ when $X$ is connected, and that $t^{n-1}c_{n*}=0$ on the torsion subgroup $\Tor H_n^\mathbb{R}(X)$ without any connectedness assumption. For a finite rack, under our sign conventions, the rationalized classifying map in degree $n$ is given by $(-1)^n$ times the canonical projection from the $n$-fold tensor power of the orbit module to its $n$-th exterior power. For an arbitrary rack of finite type, we determine $H_2^{\mathrm{gr}}(\As(X);\mathbb{Z}[1/t])$. We also derive low-dimensional applications to symplectic and Alexander structures.
\end{abstract}


\section{Introduction}

Rack and quandle homology are fundamental tools connecting self-distributive algebraic structures with knot theory \cite{CJKS,FRS}. The purpose of this paper is to describe, in terms of the type and orbit structure of a rack, the image of the canonical map from rack homology to the group homology of its associated group.

A \emph{rack} is a set $X$ equipped with a binary operation $x\lhd y$ such that every right translation $R_y(x)=x\lhd y$ is bijective and the right self-distributive law
$(x\lhd y)\lhd z=(x\lhd z)\lhd(y\lhd z)$ holds. A rack satisfying $x\lhd x=x$ is called a \emph{quandle}. Its associated group is defined by
\begin{equation}
 \As(X)=\bigl\langle e_x\ (x\in X)\ \bigm|\
 e_xe_y=e_y e_{x\lhd y}\bigr\rangle
 \label{eq:associated-group-intro}
\end{equation}
and acts on $X$ from the right by $x\cdot e_y=x\lhd y$. We write $\OO(X)$ for the set of orbits of this action and call $X$ \emph{connected} if $|\OO(X)|=1$.

Fenn--Rourke--Sanderson constructed, for each rack $X$, a connected cubical CW complex $BX$ having one $n$-cube for each $(x_1,\ldots,x_n)\in X^n$ \cite{FRSTrunks,FRS}. With suitable choices of orientations, its cellular chain complex is naturally identified with the rack chain complex, and
\[
        \pi_1(BX)\cong\As(X),
        \qquad
        H_n(BX;\ZZ)\cong H_n^\R(X).
\]
For the construction of quandle spaces adapted to quandle homology, see \cite{NosakaQuandleSpace,NosakaHomotopical}. Even when $X$ is a quandle, the notation $BX$ in this paper always denotes the rack space of the underlying rack. The canonical classifying map induced by the isomorphism on fundamental groups,
\begin{equation}
        c\colon BX\longrightarrow K(\As(X),1)
        \label{eq:classifying-map-intro}
\end{equation}
induces homomorphisms $c_{n*}\colon H_n^\R(X)\to H_n^{\gr}(\As(X))$. This is the canonical map from $BX$ to its first Postnikov stage and measures the part of rack homology detected by the group homology of the associated group.

The \emph{type} of a rack $X$ is defined by
\[
 \Type(X)=\min\{t>0\mid R_y^t=\operatorname{id}_X\ (\forall y\in X)\}
\]
when such a positive integer exists; in that case, $X$ is said to be of finite type. In general, this invariant differs from the rack rank, which is defined as the order of the kink map $x\mapsto x\lhd x$ \cite{ElhamdadiNelsonRank}. For a connected quandle of finite type, the author proved that $t c_{n*}=0$ in degrees $2$ and $3$, and subsequently asked whether this extends to all degrees and to nonconnected quandles \cite[Theorem~6.1 and the problem immediately following it]{NosakaAdjoint}. Our first main result extends the vanishing theorem to all degrees for connected racks of finite type and gives a uniform bound on the torsion part without assuming connectedness.

\Needspace{13\baselineskip}
\begin{theorem}[Main theorem]
\label{thm:main}
Let $X$ be a rack with $\Type(X)=t<\infty$.
\begin{enumerate}[label=(\arabic*)]
\item If $X$ is connected, then, for every $n\ge2$,
\[
        t\,c_{n*}=0
        \colon H_n^\R(X)\longrightarrow H_n^{\gr}(\As(X)).
\]
\item Without assuming that $X$ is connected, for every $n\ge1$ and every
$\alpha\in\Tor H_n^\R(X)$,
\[
        t^{n-1}c_{n*}(\alpha)=0.
\]
\end{enumerate}
\end{theorem}

The key ingredient in the proof is an edge-replacement identity for the explicit chain $\kappa_*$ representing the classifying map. For each orbit $X_\lambda$ and each $x\in X_\lambda$, the element $s_\lambda=(e_x)^t$ is central. On chains whose first coordinate lies in $X_\lambda$, the chain $t\kappa_n$ is related by a chain homotopy to the negative of the corresponding central-suspension term determined by $s_\lambda$. This identity is obtained from the shuffle product and a fan chain in the homogeneous bar complex. For a connected quandle, a degeneracy insertion eliminates the suspension term; a general connected rack is reduced to this case by passing to the quandle reflection, which does not change the associated group. In the nonconnected case, we use the decomposition by the orbit of the first coordinate and induction on the degree.

Whereas the main theorem captures the torsion behavior governed by the type, the classifying map on the free part of a finite rack depends, rationally, only on the orbit set. Let $A_X=\ZZ[\OO(X)]$ be the free $\ZZ$-module on the orbit set. The standard rational identifications are
\[
 H_n^\R(X;\QQ)\cong(A_X\otimes\QQ)^{\otimes n},
 \qquad
 H_n^{\gr}(\As(X);\QQ)\cong\Lambda^n(A_X\otimes\QQ).
\]
Under these identifications, Theorem~\ref{thm:finite-rack-free-image} identifies the induced map as
\[
        c_{n*}\otimes\operatorname{id}_{\QQ}=(-1)^n\operatorname{Alt}_n
        \colon (A_X\otimes\QQ)^{\otimes n}
        \longrightarrow\Lambda^n(A_X\otimes\QQ).
\]
Thus, if $|X| < \infty$ and $m=|\OO(X)|$, the image of the induced map on torsion-free quotients has rank $\binom{m}{n}$. Moreover, for a finite rack with homogeneous orbits, the explicit cycles of Litherland--Nelson yield a concrete description of the image of a finite-index sublattice of the free quotient \cite{EtingofGrana,LitherlandNelson}. In this sense, the type controls the torsion behavior of the classifying map, whereas the orbit structure controls its rational free part.

In degree $2$, the edge-replacement identity and the orbit factorization show that the orbit-degree homomorphism, which sends each generator $e_x$ to its orbit, induces the canonical isomorphism established in Theorem~\ref{thm:H2-away-from-type}:
\[
 H_2^{\gr}\bigl(\As(X);\ZZ[1/t]\bigr)
 \cong
 \Lambda^2_{\ZZ[1/t]}\bigl(\ZZ[1/t][\OO(X)]\bigr).
\]
We also prove that the composite of the classifying map with any group homomorphism annihilating all elements $(e_x)^t$ is itself annihilated by $t$ on the whole of rack homology. In particular, this applies to the natural homomorphism from $\As(X)$ to the inner automorphism group $\Inn(X)=\langle R_x\mid x\in X\rangle$. Applied to Coxeter quandles, it gives another proof that $2H_2^{\gr}(W)=0$ for every finite-rank Coxeter group $W$; see Corollary~\ref{cor:coxeter-H2} and \cite{AkitaCoxeter,HowlettCoxeter}. We further determine $\pi_2(BX)$ for coefficient-contracted symplectic quandles over finite fields (Theorem~\ref{thm:symp-pi2}) and explicitly compute $H_2^{\gr}(\As(X))$ for connected Alexander quandles for which $\mathrm{id}_M-T$ is invertible (Theorem~\ref{thm:alexander-h2-associated}). The relevant examples are defined in the sections in which they are used.

The paper is organized as follows. Section~2 reviews the rack and quandle chain complexes, an explicit chain model for the classifying map, and the quandle reflection. Section~3 establishes the edge-replacement identity by means of the shuffle product and fan chains, and Section~4 proves the main theorem. Section~5 describes the free part for finite racks. Section~6 treats homomorphisms that kill the central elements arising from the type, the application to Coxeter groups, and the second group homology after inverting $t$. Sections~7 and~8 contain low-dimensional calculations for symplectic and Alexander quandles, respectively.

\section{Chain complexes and a chain model for the classifying map}

In this section, we fix the conventions and tools needed to treat the classifying map at the chain level. This section is primarily expository. In Subsection~\ref{subsec:rack-quandle-complexes}, we fix the rack and quandle chain complexes and their sign conventions, and introduce the homogeneous bar complex and its coinvariant complex. In Subsection~\ref{subsec:kappa-chain}, we rewrite Kabaya's explicit chain in homogeneous form and define the classifying chain $\kappa_\bullet$. Finally, Subsection~\ref{subsec:quandle-reflection} introduces the quandle reflection, which reduces a general rack to a quandle, and establishes the naturality of the classifying chain. We follow Fenn--Rourke--Sanderson for rack spaces and classifying maps \cite{FRS}.

\subsection{Rack and quandle complexes and the homogeneous bar complexes of groups}
\label{subsec:rack-quandle-complexes}

For a rack $X$, set $X^0=\{*\}$ and define $C_n^\R(X)=\ZZ[X^n]$ for $n\ge0$. In particular, $C_0^\R(X)=\ZZ[*]\cong\ZZ$. For $n\ge1$, the boundary homomorphism $\partial_n^\R\colon C_n^\R(X)\to C_{n-1}^\R(X)$ is defined by
\[
\partial_n^\R(x_1,\ldots,x_n)
=
\sum_{i:\,1\le i\le n}(-1)^i
\bigl(d_i^1(x_1,\ldots,x_n)-d_i^0(x_1,\ldots,x_n)\bigr),
\]
where
\[
\begin{aligned}
 d_i^0(x_1,\ldots,x_n)
 &= (x_1,\ldots,\widehat{x_i},\ldots,x_n),\\
 d_i^1(x_1,\ldots,x_n)
 &= (x_1\lhd x_i,\ldots,x_{i-1}\lhd x_i,x_{i+1},\ldots,x_n).
\end{aligned}
\]
For $i=1$, there are no preceding coordinates, so $d_1^1=d_1^0$. Hence $\partial_1^\R=0$, and with our convention $\partial_2^\R(x,y)=(x\lhd y)-(x)$.

Suppose that $X$ is a quandle. For $n\ge2$, let $C_n^\D(X)$ be the subgroup generated by those $(x_1,\ldots,x_n)$ for which $x_i=x_{i+1}$ for some $i$, and put $C_0^\D(X)=C_1^\D(X)=0$. These groups form a subcomplex. The quandle chain complex is $C_n^\Q(X)=C_n^\R(X)/C_n^\D(X)$, and the corresponding homology groups are denoted by $H_n^\R(X)$ and $H_n^\Q(X)$, respectively. For the standard definitions of rack and quandle homology and their applications to quandle cocycle invariants, see, for example, Carter--Jelsovsky--Kamada--Langford--Saito \cite{CJKS}.

The presentation~\eqref{eq:associated-group-intro} is equivalent to $e_y^{-1}e_xe_y=e_{x\lhd y}$. Consequently, $e_{x\cdot g}=g^{-1}e_xg$ for every $g\in\As(X)$. For an orbit $\lambda\in\OO(X)$, we denote the corresponding subset of $X$ by $X_\lambda$.

Meanwhile, for a group $G$, put $\tB_m(G)=\ZZ[G^{m+1}]$ for $m\ge0$, and write its generators as $\langle g_0,\ldots,g_m\rangle$. For $m\ge1$, define
\[
\td_m\langle g_0,\ldots,g_m\rangle
        =
        \sum_{r:\,0\le r\le m}(-1)^r
        \langle g_0,\ldots,\widehat{g_r},\ldots,g_m\rangle
\]
and set $\td_0=0$. The group $G$ acts by simultaneous left translation. Regard $\ZZ$ as a trivial right $\ZZ G$-module and set $B_m^{\gr}(G)=\ZZ\otimes_{\ZZ G}\tB_m(G)$. We write $\varpi_m\colon\tB_m(G)\to B_m^{\gr}(G)$ for the natural projection and $\partial_m^{\gr}$ for the induced boundary.

The standard $\ZZ G$-linear isomorphism with the unnormalized inhomogeneous bar complex is
\[
\langle g_0,\ldots,g_m\rangle
 \longleftrightarrow
 g_0[\,g_0^{-1}g_1\mid g_1^{-1}g_2\mid\cdots\mid g_{m-1}^{-1}g_m\,].
\]
Thus, after passing to coinvariants, it becomes
\[
\varpi_m\langle g_0,\ldots,g_m\rangle
 \longleftrightarrow
 [\,g_0^{-1}g_1\mid g_1^{-1}g_2\mid\cdots\mid g_{m-1}^{-1}g_m\,].
\]
The normalized inhomogeneous bar complex is obtained from this unnormalized complex by quotienting out the degenerate chains containing the identity element.

\subsection{The classifying chain \texorpdfstring{$\kappa_n$}{kappa-n}}
\label{subsec:kappa-chain}

Kabaya constructed an explicit chain map from the rack complex to the inhomogeneous bar complex, and the author  observed that the induced homomorphism on homology is the classifying map~\eqref{eq:classifying-map-intro} \cite[\S8.4]{Kabaya}; see also \cite[\S6]{NosakaAdjoint}. We now put $G=\As(X)$, rewrite this chain in homogeneous form, and fix a lift before passing to coinvariants.

In degree zero, set $\tkap_0(*)=\langle1\rangle$ and $\kappa_0(*)=\varpi_0\langle1\rangle$. Let $n\ge1$ and $\mathbf x=(x_1,\ldots,x_n)\in X^n$. For $I=\{i_1<\cdots<i_k\}\subset\{1,\ldots,n\}$, put $g_I(\mathbf x)=e_{x_{i_1}}\cdots e_{x_{i_k}}$ and $g_\varnothing(\mathbf x)=1$, with the product taken in increasing order of the indices. Define the chain $\tkap_n(\mathbf x)\in\tB_n(G)$ before passing to coinvariants by
\begin{equation}
\tkap_n(\mathbf x)
=
(-1)^n
\sum_{\sigma\in S_n}\sgn(\sigma)
\Bigl\langle
        g_\varnothing,
        g_{\{\sigma(1)\}},
        g_{\{\sigma(1),\sigma(2)\}},
        \ldots,
        g_{\{1,\ldots,n\}}
\Bigr\rangle
\label{eq:kappa-definition}
\end{equation}
and set $\kappa_n(\mathbf x)=\varpi_n(\tkap_n(\mathbf x))\in B_n^{\gr}(G)$. For example, in the normalized inhomogeneous bar complex,
\[
 \kappa_1(x)=-[e_x],
 \qquad
 \kappa_2(x,y)=[e_x\mid e_y]-[e_y\mid e_{x\lhd y}].
\]
Under the standard homogeneous--inhomogeneous identification, $\tkap_n$ is $(-1)^n$ times Kabaya's chain; this factor compensates for the difference between his boundary convention and ours.

\begin{proposition}[Boundary formula for the classifying chain]
\label{prop:kappa-boundary}
Let $X$ be a rack. For every $n\ge1$, before passing to coinvariants one has
\begin{equation}
\td_n\tkap_n(\mathbf x)
=
\sum_{i:\,1\le i\le n}(-1)^{i+1}\tkap_{n-1}(d_i^0\mathbf x)
+
\sum_{i:\,1\le i\le n}(-1)^i e_{x_i}\,\tkap_{n-1}(d_i^1\mathbf x)
\label{eq:kappa-homogeneous-boundary}
\end{equation}
Consequently, after passing to coinvariants,
\begin{equation}
        \partial_n^{\gr} \circ  \kappa_n
        =
        \kappa_{n-1} \circ \partial_n^\R
\label{eq:kappa-chain-map}
\end{equation}
Moreover, if $X$ is a quandle, then $\kappa_\bullet$ maps the degenerate subcomplex $C_*^\D(X)$ to zero.
\end{proposition}

\begin{proof}
The chain $T_n(\mathbf x):=(-1)^n\widetilde\kappa_n(\mathbf x)$ is the oriented chain obtained by mapping the vertex $I\subset\{1,\ldots,n\}$ of the standard triangulation of the $n$-cube to $g_I(\mathbf x)$. Here the top-dimensional simplex indexed by $\sigma\in S_n$ is assigned the orientation $\sgn(\sigma)$ relative to the standard orientation of the cube. With this convention, the interior faces cancel in pairs, and
\[
 \td T_n=\sum_{i:\,1\le i\le n}(-1)^{i-1}\bigl(T_{i,1}-T_{i,0}\bigr).
\]
The $0$-face in the $i$th coordinate is $T_{n-1}(d_i^0\mathbf x)$. On the other hand, after relabeling the indices in an order-preserving way, one has $g_{J\cup\{i\}}(\mathbf x)=e_{x_i}\,g_J(d_i^1\mathbf x)$ for every $J\subset\{1,\ldots,n\}\setminus\{i\}$. Hence the corresponding $1$-face is $e_{x_i}T_{n-1}(d_i^1\mathbf x)$. Substituting $T_m=(-1)^m\widetilde\kappa_m$ gives~\eqref{eq:kappa-homogeneous-boundary}; passing to coinvariants yields~\eqref{eq:kappa-chain-map}. This is the homogeneous form of Kabaya's classifying chain \cite[\S8.4]{Kabaya}.

Finally, suppose that $x_r=x_{r+1}$. The map
$\sigma\mapsto(r\ r+1)\circ\sigma$ is a fixed-point-free involution of $S_n$. In each pair of corresponding terms, the index sets representing the vertices are obtained from one another by interchanging $r$ and $r+1$; since $x_r=x_{r+1}$, the two terms determine the same homogeneous simplex. Their permutation signs are opposite, so the two terms cancel. Thus every degenerate chain is mapped to zero.
\end{proof}

Kabaya's cubical boundary is the negative of our $\partial^\R$. If a rack generator in degree $n$ is identified with $(-1)^n$ times the corresponding cell, then $\kappa_\bullet$ is the cellular chain map induced by the classifying map. We shall therefore identify $\kappa_{n*}$ with $c_{n*}$ throughout \cite[\S8.4]{Kabaya}. The last assertion of Proposition~\ref{prop:kappa-boundary} also shows that $\kappa_\bullet$ factors through the quandle chain complex $C_*^\Q(X)$.

\subsection{The quandle reflection of a rack}
\label{subsec:quandle-reflection}

For a rack $X$, put $\iota_X(x)=x\lhd x$, and let $\sim_{\mathrm q}$ be the rack congruence---that is, an equivalence relation compatible with the rack operation---generated by
$x\sim_{\mathrm q}\iota_X(x)$ for all $x\in X$. Set $X^{\mathrm q}:=X/{\sim_{\mathrm q}}$ and write $q_X\colon X\to X^{\mathrm q}$ for the quotient map. This is the universal quandle quotient of $X$ and is also called the quandle reduction
\cite{FennRourke}; see also \cite[Proposition~5.1]{TanakaTaniguchi}. Some references use the inverse kink map, but it generates the same congruence. Indeed, the quotient operation $[x]\lhd[y]=[x\lhd y]$ makes $X^{\mathrm q}$ a quandle, and every rack homomorphism from $X$ to a quandle factors uniquely through $q_X$. Thus $q_X$ is the reflection from racks to quandles.

The quandle reduction does not change the associated group. Substituting $y=x$ into the defining relation of $\As(X)$ gives $e_x=e_{x\lhd x}$. Hence the canonical rack homomorphism $x\mapsto e_x$ is constant on $\sim_{\mathrm q}$ and induces a homomorphism $\As(X^{\mathrm q})\to\As(X)$ given by $e_{[x]}\mapsto e_x$. Therefore the homomorphism
\[
 \Phi_X:=\As(q_X)\colon
 \As(X)\xrightarrow{\cong}\As(X^{\mathrm q}),
 \qquad e_x\longmapsto e_{[x]},
\]
is an isomorphism, with inverse given by the homomorphism above.

\begin{proposition}
\label{prop:quandle-reflection}
\begin{enumerate}[label=(\arabic*)]
\item If $\Type(X)=t<\infty$, then $d:=\Type(X^{\mathrm q})$ is finite and divides $t$. Moreover, if $X$ is connected, then so is $X^{\mathrm q}$.
\item Let $q_{X,n}\colon C_n^\R(X)\to C_n^\R(X^{\mathrm q})$ be the chain map obtained by applying $q_X$ to each coordinate. To distinguish the two targets, write the classifying chains as $\kappa_n^X$ and $\kappa_n^{X^{\mathrm q}}$. Then the bar-chain map induced by $\Phi_X$ satisfies
\begin{equation}
 (\Phi_X)_\#\kappa_n^X
 =\kappa_n^{X^{\mathrm q}}q_{X,n}.
 \label{eq:quandle-reflection-classifying}
\end{equation}
\end{enumerate}
\end{proposition}

\begin{proof}
If $R_y^t=\operatorname{id}_X$, then $R_{[y]}^t=\operatorname{id}$ on the quotient, and hence $d\mid t$. Since $q_X$ is surjective and equivariant with respect to $\Phi_X$, the quandle $X^{\mathrm q}$ is connected whenever $X$ is connected.

For every $I\subset\{1,\ldots,n\}$, one has $\Phi_X\bigl(g_I^X(\mathbf x)\bigr)=g_I^{X^{\mathrm q}}(q_{X,n}\mathbf x)$. Substituting this into~\eqref{eq:kappa-definition} gives~\eqref{eq:quandle-reflection-classifying}.
\end{proof}

\section{The edge-replacement identity via shuffles}

In this section, we construct the chain homotopy underlying the main theorem for racks of finite type. In Subsection~\ref{subsec:shuffle-and-suspension}, we review the shuffle product and central suspension in the homogeneous bar complex. In Subsection~\ref{subsec:fan-edge-replacement}, we identify the central elements arising from the type and prove the edge-replacement identity by means of a fan chain. Since shuffles involving noncentral elements do not in general descend to $B_*^{\gr}(G)$, all calculations are carried out before passing to coinvariants.

\subsection{Shuffles and central suspension}
\label{subsec:shuffle-and-suspension}

For a group $G$, set $E_mG=G^{m+1}$ with its standard simplicial
structure and identify $\ZZ[E_\bullet G]$ with $(\tB_*(G),\td)$. Let
$\operatorname{sh}$ be the standard Eilenberg--Zilber shuffle and let
$\mu_\#$ be the chain map induced by coordinatewise multiplication
$\mu((u_i),(v_i))=(u_iv_i)$. For $U\in\tB_p(G)$ and $V\in\tB_r(G)$, set
\begin{equation}
 U\diamond V
 :=\mu_\#\operatorname{sh}(U\otimes V)
 \in\tB_{p+r}(G).
 \label{eq:shuffle-definition}
\end{equation}
Write $\operatorname{Sh}(p,r)$ for
the set of $(p,r)$-shuffles. For generators
$U=\langle u_0,\ldots,u_p\rangle$ and
$V=\langle v_0,\ldots,v_r\rangle$, if the $k$th vertex of the lattice path
associated with such a shuffle is $(i_k,j_k)$, then the corresponding term is
$\langle u_{i_0}v_{j_0},\ldots,u_{i_{p+r}}v_{j_{p+r}}\rangle$, with the usual
shuffle sign.

\begin{lemma}[Shuffle Leibniz rule]
For every $U\in\tB_p(G)$ and $V\in\tB_r(G)$,
\[
        \td_{p+r}(U\diamond V)
        =
        (\td_p U)\diamond V
        +(-1)^p U\diamond(\td_r V).
\]
\end{lemma}

\begin{proof}
Both $\operatorname{sh}$ and $\mu_\#$ are chain maps \cite{EilenbergZilber}; see also \cite[Chapter~VIII, \S8]{MacLaneHomology}; the assertion follows immediately from~\eqref{eq:shuffle-definition} and the sign convention for the tensor-product differential.
\end{proof}

For $k\in G$, set $kV=\langle kv_0,\ldots,kv_r\rangle$ and $U^k=\langle k^{-1}u_0k,\ldots,k^{-1}u_pk\rangle$. Then, in coinvariants,
\begin{equation}
        \varpi\bigl(U\diamond(kV)\bigr)
        =\varpi\bigl(U^k\diamond V\bigr)
        \label{eq:shuffle-conjugation}
\end{equation}
Indeed, multiplying every vertex on the left by $k^{-1}$ transforms the simplex on the left into the corresponding simplex on the right. Thus $\diamond$ does not in general descend to the coinvariant complex, although it does when the first variable is invariant under conjugation. We shall use this descended operation only for central suspension; all other shuffle calculations will be carried out in $\tB_*(G)$ before passing to coinvariants.

\begin{lemma}[Central suspension]
\label{lem:center-suspension}
Let $s\in Z(G)$ and put $L_s=\langle1,s\rangle$. For
$A\in B_m^{\gr}(G)$, choose a representative
$\widetilde A\in\tB_m(G)$ and define
\[
 \Sigma_s(A)
 :=\varpi_{m+1}(L_s\diamond\widetilde A).
\]
This definition is independent of the representative, and
\[
 \partial_{m+1}^{\gr}\Sigma_s(A)
 =-\Sigma_s(\partial_m^{\gr}A).
\]
Thus $\Sigma_s$ induces a homomorphism of degree $1$
\[
 (\Sigma_s)_*\colon H_m^{\gr}(G)\longrightarrow H_{m+1}^{\gr}(G).
\]

If $s=1$, then the induced homomorphism $(\Sigma_1)_*$ is zero.
\end{lemma}

\begin{proof}
Since the relations defining the coinvariant complex are generated by chains of the form $kV-V$, it suffices to check invariance under replacing a generator $V$ by $kV$. If the representative is replaced by $k\widetilde A$, then~\eqref{eq:shuffle-conjugation} gives
$\varpi(L_s\diamond k\widetilde A)=\varpi(L_s^k\diamond\widetilde A)$. Since $s$ is central, $L_s^k=L_s$, and therefore $\Sigma_s$ is well defined. Moreover, $\td L_s=\langle s\rangle-\langle1\rangle$, so the Leibniz rule gives
\[
 \partial\Sigma_s(A)
 =\varpi\bigl((\langle s\rangle-\langle1\rangle)\diamond\widetilde A\bigr)
  -\Sigma_s(\partial A)
 =-\Sigma_s(\partial A),
\]
where the first term vanishes in coinvariants because $\varpi(s\widetilde A-\widetilde A)=0$.

When $s=1$, let $U_2=\langle1,1,1\rangle$, so that $\td U_2=L_1$. The map $H_m(A):=\varpi(U_2\diamond\widetilde A)$ is well defined, and the Leibniz rule gives
$\partial H_m=\Sigma_1+H_{m-1}\partial$. Hence $(\Sigma_1)_*=0$.
\end{proof}

\subsection{Central elements arising from the type, fan chains, and edge replacement}
\label{subsec:fan-edge-replacement}

We first record the relation between the type and central elements of the associated group in the form needed below.

\begin{lemma}
\label{lem:central-type}
Let $X$ be a rack, and suppose that $t>0$ satisfies $R_y^t=\operatorname{id}_X$ for every $y\in X$. For an orbit $\lambda\in\OO(X)$ and an element $x\in X_\lambda$, put $s_\lambda=(e_x)^t$. Then $s_\lambda$ is independent of the choice of $x\in X_\lambda$ and lies in the center $Z(\As(X))$.
\end{lemma}

\begin{proof}
The author proved the same assertion for connected quandles \cite[Lemma~3.5]{NosakaAdjoint}. The following direct calculation does not use the quandle condition. Indeed, 
one has 
\[ (e_x)^{-t}e_b(e_x)^t=e_{b\cdot(e_x)^t}=e_{R_x^t(b)}=e_b, \qquad x,b \in X .\]
Thus $(e_x)^t$ commutes with every generator $e_b$ and is therefore central. If $x'=x\cdot g$, then $e_{x'}=g^{-1}e_xg$, and centrality gives $(e_{x'})^t=g^{-1}(e_x)^tg=(e_x)^t$.
\end{proof}

Henceforth, let $X$ be a rack, put $G=\As(X)$, and fix $t>0$ such that
$R_y^t=\operatorname{id}_X$ for every $y\in X$. For
$\lambda\in\OO(X)$ and $n\ge1$, let $C_n^{\R,\lambda}(X)$ be the
subgroup generated by tuples whose first coordinate lies in $X_\lambda$,
and set $C_0^{\R,\lambda}(X)=0$. Since the boundary preserves the orbit of
the first coordinate, $C_*^{\R,\lambda}(X)$ is a subcomplex.

Let $n\ge1$ and let $\mathbf x=(x_1,\ldots,x_n)$ have first coordinate in $X_\lambda$. Put $a=e_{x_1}$ and $s=s_\lambda=a^t$, and define the fan chain by
\[
        \tF_a=
        \sum_{r:\,0\le r\le t-2}\langle a^r,a^{r+1},s\rangle
        \in \tB_2(G),
\]
where the sum is empty when $t=1$. Also define the chain representing the subdivided first edge by
\[
        \tE_a=\sum_{r:\,0\le r\le t-1}\langle a^r,a^{r+1}\rangle.
\]
The chain $L_s=\langle1,s\rangle$ is the one used in
Lemma~\ref{lem:center-suspension}. Taking boundaries gives the telescoping
identity
\begin{equation}
        \td_2\tF_a=\tE_a-L_s,
        \label{eq:fan-boundary}
\end{equation}
where the empty-sum convention includes the case $t=1$.

Write $p_n(x_1,\ldots,x_n)=(x_2,\ldots,x_n)$ for deletion of the first coordinate, with the convention $p_1(x)=*\in X^0$. Put $\widehat{\mathbf x}=p_n\mathbf x$ and, for every $n\ge1$, define
\[
 \tP_n^\lambda(\mathbf x)
 :=\tF_a\diamond\tkap_{n-1}(\widehat{\mathbf x})\in\tB_{n+1}(G),
 \qquad
 K_n^\lambda(\mathbf x):=-\varpi_{n+1}\bigl(\tP_n^\lambda(\mathbf x)\bigr).
\]
Thus $K_n^\lambda\colon C_n^{\R,\lambda}(X)\to B_{n+1}^{\gr}(G)$. Since $\tF_a\diamond\tkap_0(*)=\tF_a$, in degree one we have $\tP_1^\lambda(x)=\tF_{e_x}$ and $K_1^\lambda(x)=-\varpi_2(\tF_{e_x})$. Thus no separate convention is needed when $n=2$.

\begin{lemma}[Shuffle along the first edge]
\label{lem:first-edge-shuffle}
For every $n\ge1$ and every $(x_1,\ldots,x_n)\in X^n$,
\[
\langle1,e_{x_1}\rangle\diamond
 \tkap_{n-1}(x_2,\ldots,x_n)
 =-\tkap_n(x_1,\ldots,x_n).
\]
\end{lemma}

\begin{proof}
Consider the following insertion map $\operatorname{Sh}(1,n-1)\times S_{n-1}\to S_n$. Regard $\tau\in S_{n-1}$ as a permutation of the indices $2,\ldots,n$. For a shuffle that takes $r$ steps in the second direction before taking the step in the first direction, insert the index $1$ after the first $r$ entries of $\tau$ to obtain $\sigma\in S_n$. This gives a bijection. The shuffle sign is $(-1)^r$, and $\sgn(\sigma)=(-1)^r\sgn(\tau)$. The corresponding vertex sequence agrees with the $\sigma$-term in~\eqref{eq:kappa-definition}; hence its coefficient on the left-hand side is
$(-1)^{n-1}(-1)^r\sgn(\tau)=(-1)^{n-1}\sgn(\sigma)$. By contrast, the coefficient of the $\sigma$-term in $\tkap_n$ is $(-1)^n\sgn(\sigma)$, proving the assertion.
\end{proof}

\begin{proposition}[Edge-replacement identity]
\label{prop:edge-replacement}
Let $X$ be a rack, and let $t>0$ satisfy $R_y^t=\operatorname{id}_X$ for every $y\in X$. For every $\lambda\in\OO(X)$ and every $n\ge2$, the following identity holds on $C_n^{\R,\lambda}(X)$:
\begin{equation}
        t\kappa_n+
        \Sigma_{s_\lambda}\kappa_{n-1}p_n
        =
        \partial_{n+1}^{\gr}K_n^\lambda
        +K_{n-1}^\lambda\partial_n^\R.
\label{eq:edge-replacement}
\end{equation}
\end{proposition}

\begin{proof}
Let $\mathbf x=(x_1,\ldots,x_n)$ be a generator with $x_1\in X_\lambda$, and put $a=e_{x_1}$, $s=s_\lambda$, and $\widehat{\mathbf x}=p_n\mathbf x$. The Leibniz rule and~\eqref{eq:fan-boundary} give
\begin{equation}
 \td\tP_n^\lambda(\mathbf x)
 =(\tE_a-L_s)\diamond\tkap_{n-1}(\widehat{\mathbf x})
 +\tF_a\diamond\td\tkap_{n-1}(\widehat{\mathbf x}).
 \label{eq:P-boundary-start}
\end{equation}
For each $r$, simultaneously left-translate every vertex of
$\langle a^r,a^{r+1}\rangle\diamond
\tkap_{n-1}(\widehat{\mathbf x})$ by $a^{-r}$ and apply
Lemma~\ref{lem:first-edge-shuffle}. Then
\[
 \varpi\bigl(\tE_a\diamond\tkap_{n-1}(\widehat{\mathbf x})\bigr)
 =-t\kappa_n(\mathbf x).
\]
Moreover,
$\varpi(L_s\diamond\tkap_{n-1}(\widehat{\mathbf x}))=\Sigma_s\kappa_{n-1}(p_n\mathbf x)$.

Apply Proposition~\ref{prop:kappa-boundary} to the remaining term.
For $j=1,\ldots,n-1$, put $i=j+1$. The $0$-face term of the
tail has coefficient $(-1)^{j+1}=(-1)^i$ and gives
$\tP_{n-1}^\lambda(d_i^0\mathbf x)$. For a $1$-face term,
\eqref{eq:shuffle-conjugation} and
$e_{x_i}^{-1}ae_{x_i}=e_{x_1\lhd x_i}$ give
\[
\varpi\bigl(\tF_a\diamond e_{x_i}V\bigr)
=
\varpi\bigl(\tF_{e_{x_i}^{-1}ae_{x_i}}\diamond V\bigr),
\]
so its coefficient is $(-1)^j=(-1)^{i+1}$ and it gives
$\tP_{n-1}^\lambda(d_i^1\mathbf x)$. Consequently,
\begin{align*}
&\varpi\bigl(
\tF_a\diamond\td\tkap_{n-1}(\widehat{\mathbf x})
\bigr)\\
&\quad=
\sum_{i:\,2\le i\le n}(-1)^i
\varpi\bigl(\tP_{n-1}^\lambda(d_i^0\mathbf x)\bigr)
+
\sum_{i:\,2\le i\le n}(-1)^{i+1}
\varpi\bigl(\tP_{n-1}^\lambda(d_i^1\mathbf x)\bigr)\\
&\quad=
-\varpi\bigl(
\tP_{n-1}^\lambda(\partial_n^\R\mathbf x)
\bigr).
\end{align*}
Here the $i=1$ contribution to the rack boundary is zero because
$d_1^1=d_1^0$. Substituting these identities into
\eqref{eq:P-boundary-start} and using
$K_m^\lambda=-\varpi\tP_m^\lambda$ yields
\eqref{eq:edge-replacement}.
\end{proof}

\section{Proof of the main theorem}

We now prove Theorem~\ref{thm:main}. We first record deletion of the first coordinate and the degeneracy insertion for quandles, then prove the vanishing theorem for connected quandles and pass to connected racks by means of the quandle reflection. Finally, we establish the torsion statement without connectedness by decomposing according to the orbit of the first coordinate and arguing by induction on the degree.

\begin{lemma}
\label{lem:auxiliary-maps}
For every rack $X$, deletion of the first coordinate, $p_n(x_1,\ldots,x_n)=(x_2,\ldots,x_n)$ with $p_1(x)=*$, satisfies
\begin{equation}
 \partial_{n-1}^\R p_n=-p_{n-1}\partial_n^\R
 \label{eq:p-anti-chain}
\end{equation}
and therefore induces $p_*\colon H_n^\R(X)\to H_{n-1}^\R(X)$.

If $X$ is a quandle and $j_n(y_1,\ldots,y_{n-1})=(y_1,y_1,y_2,\ldots,y_{n-1})$, then $\partial_2^\R j_2=0$, and for $n\ge3$ one has
\begin{equation}
 \partial_n^\R j_n=-j_{n-1}\partial_{n-1}^\R.
 \label{eq:j-anti-chain}
\end{equation}
Moreover,
\[
p_nj_n=\operatorname{id},
 \qquad
 \kappa_nj_n=0.
\]
\end{lemma}

\begin{proof}
For $p$, the equality $d_1^1=d_1^0$ holds, while for $i\ge2$ one has $p_{n-1}d_i^\epsilon=d_{i-1}^\epsilon p_n$; the shift in the index gives~\eqref{eq:p-anti-chain}. For $j$, the term with $i=1$ vanishes, and the term with $i=2$ vanishes because $y_1\lhd y_1=y_1$. For $i\ge3$, one has $d_i^\epsilon j_n=j_{n-1}d_{i-1}^\epsilon$, which gives~\eqref{eq:j-anti-chain}. The last two identities follow immediately from the definitions and from the assertion on degenerate chains in Proposition~\ref{prop:kappa-boundary}.
\end{proof}

We first treat connected quandles.

\begin{proposition}
\label{prop:connected-quandle-annihilation}
Let $X$ be a connected quandle with $\Type(X)=t<\infty$. Then, for every $n\ge2$,
\[
        t\,c_{n*}=0
        \colon H_n^\R(X)\longrightarrow H_n^{\gr}(\As(X)).
\]
\end{proposition}

\begin{proof}
Let $\xi$ be a rack $n$-cycle, and put $\eta=p_n\xi$ and
$\zeta=\xi-j_n\eta$. By Lemma~\ref{lem:auxiliary-maps}, the chain $\zeta$ is a cycle, $p_n\zeta=0$, and $\kappa_n(\zeta)=\kappa_n(\xi)$.
Since $X$ has only one orbit, write the corresponding map $K_n^\lambda$ simply as $K_n$. Applying Proposition~\ref{prop:edge-replacement} to $\zeta$ and using $p_n\zeta=0$ and $\partial_n^\R\zeta=0$, we obtain
\[
 t\kappa_n(\xi)=t\kappa_n(\zeta)=\partial_{n+1}^{\gr}K_n(\zeta).
\]
Thus $t c_{n*}=0$.
\end{proof}

For a general rack, the idempotency used to eliminate the $i=2$ term in~\eqref{eq:j-anti-chain} is unavailable, so the preceding degenerate-insertion argument cannot be applied directly. We therefore pass to the quandle reflection.

\begin{proof}[Proof of Theorem~\ref{thm:main}(1)]
Let $X$ be a connected rack and put $d=\Type(X^{\mathrm q})$. By Proposition~\ref{prop:quandle-reflection}, $d\mid t$, and we recall an isomorphism $ \Phi_X$ from $\As(X)$ to $\As(X^{\mathrm q})$. Proposition~\ref{prop:connected-quandle-annihilation} and~\eqref{eq:quandle-reflection-classifying} give
\[
 (\Phi_X)_*\bigl(d\,c_{n*}^X\bigr)
 =d\,c_{n*}^{X^{\mathrm q}}(q_X)_*=0.
\]
Hence $d\,c_{n*}^X=0$ and, in particular, $t\,c_{n*}^X=0$.
\end{proof}

We next prove the torsion statement without assuming connectedness. The subcomplexes defined in Section~3 by the orbit of the first coordinate give, in positive degrees,
\begin{equation}
        C_n^\R(X)=
        \bigoplus_{\lambda\in\OO(X)}C_n^{\R,\lambda}(X),
        \qquad
        H_n^\R(X)\cong
        \bigoplus_{\lambda\in\OO(X)}
        H_n(C_*^{\R,\lambda}(X)).
\label{eq:first-orbit-decomposition}
\end{equation}

Since $\partial_2^\R(x,y)=(x\lhd y)-(x)$, the group $H_1^\R(X)$ is the free abelian group obtained by identifying elements in the same $\As(X)$-orbit. Thus
$  H_1^\R(X)\cong\ZZ[\OO(X)]$
and in particular $H_1^\R(X)$ is torsion-free.

\begin{proof}[Proof of Theorem~\ref{thm:main}(2)]
We argue by induction on $n$. The case $n=1$ follows from the isomorphism $  H_1^\R(X)\cong\ZZ[\OO(X)]$. Let $n\ge2$, and write $\alpha\in\Tor H_n^\R(X)$ as $\alpha=\sum_\lambda\alpha_\lambda$ using~\eqref{eq:first-orbit-decomposition}. Each $\alpha_\lambda$ is again torsion.

Passing to homology in Proposition~\ref{prop:edge-replacement} gives
\begin{equation}
 t\,c_{n*}(\alpha_\lambda)
 =-(\Sigma_{s_\lambda})_*c_{n-1,*}(p_*\alpha_\lambda).
 \label{eq:orbit-induction-step}
\end{equation}
Since $p_*\alpha_\lambda$ is torsion, the induction hypothesis yields
$t^{n-2}c_{n-1,*}(p_*\alpha_\lambda)=0$. Multiplying~\eqref{eq:orbit-induction-step} by $t^{n-2}$ gives $t^{n-1}c_{n*}(\alpha_\lambda)=0$. Summing over $\lambda$ proves the assertion.
\end{proof}

We note that, if $d=\Type(X^{\mathrm q})$, then the preceding proof already gives $d\,c_{n*}=0$ in the connected case. In general, $q_X$ sends torsion classes to torsion classes and $\Phi_X$ is an isomorphism of associated groups. Applying Theorem~\ref{thm:main}(2) to $X^{\mathrm q}$ and using
naturality therefore gives $d^{n-1}c_{n*}=0$ on $\Tor H_n^\R(X)$.

\section{The image on the free part for finite racks}
\label{sec:free-part-image}

We now describe the free part of the classifying map in terms of orbit data.
We first prove an integral orbit-exterior formula for an arbitrary rack, then
interpret the torsion-free quotients for a finite rack as finite-index
lattices in the orbit tensor and orbit exterior powers. Finally, when the
orbits are homogeneous, we compute the images of the explicit cycles of
Litherland--Nelson.

Set $A_X=\ZZ[\OO(X)]$, and denote by $u_\lambda$ the basis element
corresponding to $\lambda\in\OO(X)$. Regard
$\mathcal T_X=\OO(X)$ as the trivial rack with operation
$\lambda\lhd\mu=\lambda$, and let
$\pi_X\colon X\to\mathcal T_X$ be the orbit map. Define
\[
\varepsilon_X\colon\As(X)\longrightarrow A_X,
 \qquad
 \varepsilon_X(e_x)=u_{[x]}.
\]
This is precisely $\As(\pi_X)$. Upon abelianizing the presentation of the
associated group, the relations merely identify generators lying in the same
orbit. Hence $\varepsilon_X$ induces the canonical isomorphism
$\As(X)_{\ab}\cong A_X$. We also let
\[
 \operatorname{Alt}_n\colon A_X^{\otimes n}\longrightarrow\Lambda^nA_X,
 \qquad
 u_{\lambda_1}\otimes\cdots\otimes u_{\lambda_n}
 \longmapsto
 u_{\lambda_1}\wedge\cdots\wedge u_{\lambda_n},
\]
be the canonical projection onto the exterior power.

\begin{proposition}[Orbit-exterior formula]
\label{prop:orbit-exterior-formula}
For every rack $X$ and every $n\ge1$,
\begin{equation}
 (\varepsilon_X)_*c_{n*}
 =(-1)^n\operatorname{Alt}_n(\pi_X)_*
 \colon H_n^\R(X)\longrightarrow\Lambda^nA_X.
 \label{eq:orbit-exterior-formula}
\end{equation}
Here we use
$H_n^\R(\mathcal T_X)=A_X^{\otimes n}$ and
$H_n^{\gr}(A_X)=\Lambda^nA_X$.
\end{proposition}

\begin{proof}
Naturality of the classifying chain gives the chain-level identity
\[
 (\varepsilon_X)_\#\kappa_n^X
 =\kappa_n^{\mathcal T_X}(\pi_X)_n.
\]
The associated group of the trivial rack $\mathcal T_X$ is the free abelian
group $A_X$, and~\eqref{eq:kappa-definition}, evaluated on
$(\lambda_1,\ldots,\lambda_n)$, becomes the inhomogeneous bar cycle
\[
 (-1)^n\sum_{\sigma\in S_n}\sgn(\sigma)
 [u_{\lambda_{\sigma(1)}}\mid\cdots\mid u_{\lambda_{\sigma(n)}}].
\]
Under the standard isomorphism
$H_n^{\gr}(A_X)\cong\Lambda^nA_X$, this cycle represents
$(-1)^n u_{\lambda_1}\wedge\cdots\wedge u_{\lambda_n}$, proving
\eqref{eq:orbit-exterior-formula}.
\end{proof}

For the remainder of this section, let $X$ be a finite rack. For a finitely
generated abelian group $M$, put $M_{\mathrm{fr}}:=M/\Tor M$. Since the choice
of a free direct summand is not canonical in general, by the free part we
always mean this canonical torsion-free quotient. The proof below shows that
both $H_n^\R(X)$ and $H_n^{\gr}(\As(X))$ are finitely generated, and the
classifying map induces
\[
 \overline c_{n*}\colon H_n^\R(X)_{\mathrm{fr}}
 \longrightarrow H_n^{\gr}(\As(X))_{\mathrm{fr}}.
\]
This records the map induced by $c_{n*}$ on torsion-free quotients, but it
does not determine possible torsion in the image of $c_{n*}$.

\begin{theorem}[The free part for a finite rack]
\label{thm:finite-rack-free-image}
Let $X$ be a finite rack, and set $m=|\OO(X)|$. For every $n\ge1$, the maps
$(\pi_X)_*$ and $(\varepsilon_X)_*$ induce injective homomorphisms with finite
cokernels
\[
 \overline\pi_{n*}\colon
 H_n^\R(X)_{\mathrm{fr}}\lhook\joinrel\longrightarrow A_X^{\otimes n},
 \qquad
 \overline\varepsilon_{n*}\colon
 H_n^{\gr}(\As(X))_{\mathrm{fr}}\lhook\joinrel\longrightarrow\Lambda^nA_X,
\]
and these satisfy
\begin{equation}
 \overline\varepsilon_{n*}\,\overline c_{n*}
 =(-1)^n\operatorname{Alt}_n\,\overline\pi_{n*}.
 \label{eq:free-part-exterior-formula}
\end{equation}
Consequently, over $\QQ$, after identifying the source and target via
$\overline\pi_{n*}$ and $\overline\varepsilon_{n*}$,
\[
 c_{n*}\otimes\operatorname{id}_{\QQ}
 =(-1)^n\operatorname{Alt}_n
 \colon
 (A_X\otimes\QQ)^{\otimes n}
 \longrightarrow
 \Lambda^n(A_X\otimes\QQ).
\]
In particular,
\[
\begin{aligned}
 \rank H_n^\R(X)_{\mathrm{fr}}&=m^n,\\
 \rank H_n^{\gr}(\As(X))_{\mathrm{fr}}
 &=\rank\operatorname{Im}(\overline c_{n*})=\binom mn,\\
 \rank\Ker(\overline c_{n*})&=m^n-\binom mn,
\end{aligned}
\]
where $\binom mn=0$ when $n>m$. Moreover,
$\Coker(\overline c_{n*})$ is finite.

For $\boldsymbol\lambda=(\lambda_1,\ldots,\lambda_n)\in\OO(X)^n$, define
\[
\Theta_{\boldsymbol\lambda}
 :=\sum_{x_1\in X_{\lambda_1},\,\ldots,\,x_n\in X_{\lambda_n}}
 (x_1,\ldots,x_n)
 \in C_n^\R(X).
\]
Then $\Theta_{\boldsymbol\lambda}$ is a cycle, and the classes
$[\Theta_{\boldsymbol\lambda}]$ form a basis of $H_n^\R(X;\QQ)$. Hence their
images $[\Theta_{\boldsymbol\lambda}]_{\mathrm{fr}}$ form a basis of a
finite-index sublattice of $H_n^\R(X)_{\mathrm{fr}}$.
\end{theorem}

\begin{proof}
For fixed $x_i$, the map $R_{x_i}$ is a permutation preserving every orbit.
Therefore
$d_i^0\Theta_{\boldsymbol\lambda}=d_i^1\Theta_{\boldsymbol\lambda}$ for all
$i$, so $\Theta_{\boldsymbol\lambda}$ is a cycle. Moreover,
\[
(\pi_X)_*[\Theta_{\boldsymbol\lambda}]
 =\Bigl(\prod_{i:\,1\le i\le n}|X_{\lambda_i}|\Bigr)
 u_{\lambda_1}\otimes\cdots\otimes u_{\lambda_n}.
\]
Thus $(\pi_X)_*\otimes\QQ$ is surjective. Etingof--Gra\~na computed rational rack cohomology. Since the rational rack chain groups of a finite rack are
finite-dimensional, the universal coefficient theorem gives
\[
 H^n_\R(X;\QQ)
 \cong\operatorname{Hom}_{\QQ}(H_n^\R(X;\QQ),\QQ).
\]
Their formula therefore yields
$\dim_\QQ H_n^\R(X;\QQ)=m^n$
\cite[Theorem~4.2 and Corollary~4.3]{EtingofGrana}. Hence
$(\pi_X)_*\otimes\QQ$ is an isomorphism, and the orbit-sum classes form a
rational basis.

Put $G=\As(X)$. The kernel of the action homomorphism
$G\to\Inn(X)$ is central, and $\Inn(X)$ is finite. Hence $G/Z(G)$ is finite.
Schur's theorem states that if $H/Z(H)$ is finite for a group $H$, then the
commutator subgroup $[H,H]$ is finite
\cite[Theorem~10.1.4]{RobinsonGroupTheory}. Thus $[G,G]$ is finite. Since
$\varepsilon_X$ is the abelianization homomorphism,
\[
 1\longrightarrow [G,G]\longrightarrow G
 \xrightarrow{\varepsilon_X}A_X\longrightarrow1
\]
has finite kernel. The Lyndon--Hochschild--Serre spectral sequence gives
\cite[Chapter~VII, \S6]{Brown}
\[
 (\varepsilon_X)_*\otimes\QQ\colon
 H_n^{\gr}(G;\QQ)\xrightarrow{\cong}
 \Lambda^n(A_X\otimes\QQ).
\]
The same spectral sequence with integral coefficients shows that
$H_n^{\gr}(G)$ is finitely generated.

It follows that $\overline\pi_{n*}$ and $\overline\varepsilon_{n*}$ are
injective with finite cokernel. Equation~\eqref{eq:free-part-exterior-formula}
follows from Proposition~\ref{prop:orbit-exterior-formula}, and the remaining
assertions follow from the surjectivity of
$\operatorname{Alt}_n\otimes\QQ$.
\end{proof}

Following Litherland--Nelson, a finite rack $X$ is said to have
\emph{homogeneous orbits} if, whenever $a,b$ belong to the same orbit
$X_\lambda$, the number
\[
 N(a,b):=|\{y\in X\mid a\lhd y=b\}|
\]
depends only on $\lambda$, not on $a$ or $b$ \cite[\S1]{LitherlandNelson}.
In this case $N(a,b)=|X|/|X_\lambda|$; denote this integer by $N_\lambda$.

\begin{corollary}[Images of the Litherland--Nelson cycles]
Let $X$ be a finite rack with homogeneous orbits. For
$\boldsymbol\lambda=(\lambda_1,\ldots,\lambda_n)\in\OO(X)^n$, set
\[
\Psi_{\boldsymbol\lambda}
 :=\Bigl(\prod_{i:\,1\le i\le n}N_{\lambda_i}\Bigr)
 \sum_{x_1\in X_{\lambda_1},\,\ldots,\,x_n\in X_{\lambda_n}}
 (x_1,\ldots,x_n)
 \in C_n^\R(X).
\]
Then $\Psi_{\boldsymbol\lambda}$ is a cycle, and the classes
$[\Psi_{\boldsymbol\lambda}]\otimes1$, as
$\boldsymbol\lambda$ ranges over $\OO(X)^n$, form a basis of
$H_n^\R(X;\QQ)$. If
$[\Psi_{\boldsymbol\lambda}]_{\mathrm{fr}}$ denotes the image in the
torsion-free quotient, then
\[
\overline\varepsilon_{n*}\,\overline c_{n*}
 \bigl([\Psi_{\boldsymbol\lambda}]_{\mathrm{fr}}\bigr)
 =(-1)^n|X|^n
 u_{\lambda_1}\wedge\cdots\wedge u_{\lambda_n}.
\]
Thus, if the $\lambda_i$ are not all distinct, then
$\overline c_{n*}([\Psi_{\boldsymbol\lambda}]_{\mathrm{fr}})=0$. If
$\lambda_1,\ldots,\lambda_n$ are distinct, this image is nonzero. In
particular, if $n\le m$ and a total order is fixed on $\OO(X)$, the elements
\[
 \overline c_{n*}
 \bigl([\Psi_{\lambda_1,\ldots,\lambda_n}]_{\mathrm{fr}}\bigr),
 \qquad \lambda_1<\cdots<\lambda_n,
\]
generate a finite-index sublattice of
$H_n^{\gr}(\As(X))_{\mathrm{fr}}$.
\end{corollary}

\begin{proof}
Substitute
$\Psi_{\boldsymbol\lambda}=(\prod_{i:\,1\le i\le n}N_{\lambda_i})
\Theta_{\boldsymbol\lambda}$ and
$N_\lambda|X_\lambda|=|X|$ into
Theorem~\ref{thm:finite-rack-free-image}. The assertions now follow
\cite[\S3]{LitherlandNelson}.
\end{proof}

These classes form a rational basis; equivalently, their images in
the torsion-free quotient form a basis of a finite-index sublattice.
They need not form a primitive integral basis. The homogeneous-orbit
assumption is used only to obtain the particular integral
normalization of the explicit Litherland--Nelson cycles; the orbit-sum
basis and Theorem~\ref{thm:finite-rack-free-image} do not require this
assumption.

Moreover, \eqref{eq:free-part-exterior-formula} describes only the map
induced on torsion-free quotients. Thus, even when some orbit indices
coincide, the class $c_{n*}[\Psi_{\boldsymbol\lambda}]$ may be a nonzero
torsion class. In particular, if $X$ is a finite connected rack, then
\[
 H_n^{\gr}(\As(X))_{\mathrm{fr}}=0
 \quad\text{and hence}\quad
 \overline c_{n*}=0
\]
for every $n\ge2$. By contrast, Theorem~\ref{thm:main}(1) only implies
that the actual image of $c_{n*}$ is $t$-torsion; this image may nevertheless
be nonzero.

\section{Further consequences of the main theorem}

We derive three consequences of the main theorem and the edge-replacement
identity. Subsection~\ref{subsec:central-killing-composite} treats composites
with group homomorphisms that kill the central elements determined by the
type, with $\As(X)\to\Inn(X)$ as a special case.
Subsection~\ref{subsec:coxeter-application} applies the same chain computation
to Coxeter quandles and deduces that the second homology of a Coxeter group is
annihilated by $2$. Finally, Subsection~\ref{subsec:H2-away-from-type}
examines the orbit-degree homomorphism and identifies $H_2(\As(X))$ after
inverting $t$ with the exterior square on the orbit set.

\subsection{Composition with homomorphisms killing the central type elements}
\label{subsec:central-killing-composite}

\begin{proposition}[Homomorphisms killing the central type elements]
\label{prop:central-killing-composite}
Let $X$ be a rack, and suppose that
$R_y^t=\operatorname{id}_X$ for every $y\in X$. If a group homomorphism
$f\colon\As(X)\to H$ satisfies
$f((e_x)^t)=1$ for all $x\in X$, then for every $n\ge1$,
\begin{equation}
 t\,f_*c_{n*}=0
 \colon H_n^\R(X)\longrightarrow H_n^{\gr}(H).
 \label{eq:central-killing-composite}
\end{equation}
In particular, for the action homomorphism
$\rho\colon\As(X)\twoheadrightarrow\Inn(X)$,
\[
 t\,\rho_*c_{n*}=0.
\]
\end{proposition}

\begin{proof}
For $n=1$, one has $c_{1*}[x]=-[e_x]$, and hence
$t[f(e_x)]=[f(e_x)^t]=0$ in $H_1^{\gr}(H)$. Let $n\ge2$. Decompose a cycle according to the
orbit of its first entry and apply the bar-chain map induced by $f$ to Proposition~\ref{prop:edge-replacement}. This chain map commutes with shuffles, and
$f(s_\lambda)=1$, so the central-suspension term vanishes because
$(\Sigma_1)_*=0$. Therefore $t f_*c_{n*}=0$ on each orbit summand, and summing
over all orbits gives~\eqref{eq:central-killing-composite}. The last assertion
follows from $\rho((e_x)^t)=R_x^t=1$.
\end{proof}

\subsection{Application to Coxeter groups}
\label{subsec:coxeter-application}

Let $(W,S)$ be a Coxeter system of finite rank. The set of all reflections
\[
 Q_W=\{wsw^{-1}\mid w\in W,\ s\in S\}
\]
becomes a quandle under $x\lhd y=yxy$; it is called the \emph{Coxeter
quandle} \cite{AkitaCoxeter}. Since every reflection $y$ is an involution,
$R_y^2=\operatorname{id}$.

The following corollary gives another proof of a consequence of Howlett's computation of the Schur multiplier \cite{HowlettCoxeter} (In cidentally, a quandle-theoretic proof is also given in \cite[Corollary 4.4]{Akita2026}).
\begin{corollary}[The second homology of a Coxeter group]
\label{cor:coxeter-H2}
Let $(W,S)$ be a Coxeter system of finite rank. Then
\[
 2\,H_2^{\gr}(W)=0.
\]
\end{corollary}

\begin{proof}
Let $R_W\subset S$ be a complete set of representatives for the
$W$-conjugacy classes of simple reflections, and let
\[
 \phi\colon\As(Q_W)\twoheadrightarrow W,
 \qquad e_x\longmapsto x.
\]
By Proposition~\ref{prop:central-killing-composite},
$2\phi_*c_{2*}=0$, and the Hopf exact sequence shows that $c_{2*}$ is
surjective \cite[Chapter~II, \S5]{Brown}.

By Akita's results, $C_W:=\Ker\phi$ is central, and
\[
 C_W=\bigoplus_{s\in R_W}\langle e_s^2\rangle\cong\ZZ^{R_W},
 \qquad
 \As(Q_W)_{\ab}=\bigoplus_{s\in R_W}\ZZ[e_s].
\]
Moreover, the five-term exact sequence of the central extension
$1\to C_W\to\As(Q_W)\to W\to1$ contains
\begin{equation}\label{fff} 
 H_2^{\gr}(\As(Q_W)) \stackrel{\phi_*}{\longrightarrow} H_2^{\gr}(W)
 \longrightarrow C_W\longrightarrow\As(Q_W)_{\ab}
 \longrightarrow W_{\ab}\longrightarrow0
\end{equation} 
\cite[Proposition~2.4, Lemma~2.5 and Theorem~3.1]{AkitaCoxeter}. The map $C_W\to\As(Q_W)_{\ab}$ sends $e_s^2$ to $2[e_s]$, so it is
multiplication by $2$ on each direct summand and is injective. Thus,
$\phi_*$ in \eqref{fff} 
is surjective, and consequently $2H_2^{\gr}(W)=0$.
\end{proof}

\subsection{The second homology of the associated group after inverting \texorpdfstring{$t$}{t}}
\label{subsec:H2-away-from-type}

Throughout this subsection, let $X$ be a rack with
$\Type(X)=t<\infty$, and set $G=\As(X)$. We retain the notation
$A_X$, $u_\lambda$, $\mathcal T_X$, $\pi_X$, and $\varepsilon_X$ from
Section~\ref{sec:free-part-image}.

For commuting elements $g,h\in G$, write
\begin{equation}
 \langle g,h\rangle_{\mathrm P}
 :=\bigl[\,[g\mid h]-[h\mid g] \,\bigr]
 \in H_2^{\gr}(G)
 \label{eq:Pontryagin-class}
\end{equation}
for their Pontryagin product. It is alternating and satisfies
$\langle g^m,h^n\rangle_{\mathrm P}=mn\langle g,h\rangle_{\mathrm P}$ for
all integers $m,n$. Indeed, it is the image of the standard generator of
$H_2^{\gr}(\ZZ^2)=\Lambda^2\ZZ^2$ under the homomorphism
$\ZZ^2\to G$, $(a,b)\mapsto g^ah^b$, while
$(a,b)\mapsto(ma,nb)$ induces multiplication by $mn$ on
$H_2^{\gr}(\ZZ^2)$. In particular, when the first variable is a fixed central
element $s$, the map $h\mapsto\langle s,h\rangle_{\mathrm P}$ factors through
$G_{\ab}$.

Choose a representative $x_\lambda\in X_\lambda$ for each orbit, and set
$a_\lambda=e_{x_\lambda}$ and $s_\lambda=a_\lambda^t$. By
Lemma~\ref{lem:central-type}, the central element $s_\lambda$ is independent
of the choice of representative. Since the rack differential of
$\mathcal T_X$ is zero,
$H_2^\R(\mathcal T_X)=\ZZ[\OO(X)^2]\cong A_X\otimes A_X$. Define a
$\ZZ$-linear map
\[
\Lambda_X\colon A_X\otimes A_X\longrightarrow H_2^{\gr}(G),
 \qquad
 \Lambda_X(u_\lambda\otimes u_\mu)
 =\langle s_\lambda,a_\mu\rangle_{\mathrm P}.
\]
This is independent of the choice of $x_\mu$, because the second variable
depends only on its class in $G_{\ab}$ and generators belonging to the same
orbit have the same image in $G_{\ab}$.

\begin{lemma}[Orbit factorization in degree two]
\label{lem:degree-two-orbit-factorization}
For the canonical map
\[
 \operatorname{alt}\colon A_X\otimes A_X\longrightarrow\Lambda^2A_X,
 \qquad u_\lambda\otimes u_\mu\longmapsto u_\lambda\wedge u_\mu,
\]
one has
\begin{align}
 t\,c_{2*}&=\Lambda_X(\pi_X)_*
 \colon H_2^\R(X)\longrightarrow H_2^{\gr}(G),\notag\\
 (\varepsilon_X)_*c_{2*}&=\operatorname{alt}(\pi_X)_*
 \colon H_2^\R(X)\longrightarrow\Lambda^2A_X,
 \label{eq:epsilon-c2-alt}\\
 (\varepsilon_X)_*\Lambda_X&=t\,\operatorname{alt}
 \colon A_X\otimes A_X\longrightarrow\Lambda^2A_X.
 \label{eq:epsilon-Lambda-alt}
\end{align}
\end{lemma}

\begin{proof}
For a central element $s$, one has
$(\Sigma_s)_*[h]=\langle s,h\rangle_{\mathrm P}$, while
$\kappa_1(y)=-[e_y]$. Apply Proposition~\ref{prop:edge-replacement} to each
first-orbit component of a $2$-cycle and move the central-suspension term to
the other side. This gives
$t\,c_{2*}=\Lambda_X(\pi_X)_*$. Equation~\eqref{eq:epsilon-c2-alt} is the case
$n=2$ of Proposition~\ref{prop:orbit-exterior-formula}. Finally,
$\varepsilon_X(s_\lambda)=t u_\lambda$, so naturality and bilinearity of the
Pontryagin product give
\[
 (\varepsilon_X)_*\Lambda_X(u_\lambda\otimes u_\mu)
 =\langle t u_\lambda,u_\mu\rangle_{\mathrm P}
 =t\,u_\lambda\wedge u_\mu,
\]
which is~\eqref{eq:epsilon-Lambda-alt}.
\end{proof}

\begin{theorem}
\label{thm:H2-away-from-type}
The kernel of the homomorphism induced by the orbit-degree map,
\[
 (\varepsilon_X)_*\colon H_2^{\gr}(G)
 \longrightarrow H_2^{\gr}(A_X)\cong\Lambda^2A_X,
\]
is annihilated by $t^2$, and its cokernel is annihilated by $t$.
Consequently, $\varepsilon_X$ induces a canonical isomorphism
\begin{equation}
 H_2^{\gr}\bigl(\As(X);\ZZ[1/t]\bigr)
 \xrightarrow{\ \cong\ }
 \Lambda^2_{\ZZ[1/t]}\bigl(\ZZ[1/t][\OO(X)]\bigr).
 \label{eq:H2-away-from-type}
\end{equation}
Moreover,
\[
 \Ker\bigl((\varepsilon_X)_*\bigr)=\Tor H_2^{\gr}(\As(X)),
 \qquad
 t^2\Tor H_2^{\gr}(\As(X))=0.
\]
\end{theorem}

\begin{proof}
Since the classifying map induces an isomorphism on fundamental groups, the
Hopf exact sequence shows that
$c_{2*}\colon H_2^\R(X)\twoheadrightarrow H_2^{\gr}(G)$ is surjective
\cite[Chapter~II, \S5]{Brown}. Let
$y\in\Ker((\varepsilon_X)_*)$, write $y=c_{2*}(\alpha)$, and put
$M=(\pi_X)_*(\alpha)$. By
Lemma~\ref{lem:degree-two-orbit-factorization},
\[
 \operatorname{alt}(M)=0,
 \qquad
 ty=\Lambda_X(M).
\]
The kernel of $\operatorname{alt}$ is generated by
$u_\lambda\otimes u_\lambda$ and
$u_\lambda\otimes u_\mu+u_\mu\otimes u_\lambda$. By bilinearity and
alternation,
\[
 t\Lambda_X(u_\lambda\otimes u_\lambda)
 =\langle s_\lambda,s_\lambda\rangle_{\mathrm P}=0,
\]
and
\[
 t\Lambda_X(u_\lambda\otimes u_\mu+u_\mu\otimes u_\lambda)
 =\langle s_\lambda,s_\mu\rangle_{\mathrm P}
 +\langle s_\mu,s_\lambda\rangle_{\mathrm P}=0.
\]
Hence $t^2y=0$, so the kernel is annihilated by $t^2$.

Equation~\eqref{eq:epsilon-Lambda-alt} also gives
$t\Lambda^2A_X\subset\operatorname{Im}((\varepsilon_X)_*)$, so the cokernel
is annihilated by $t$. Thus $(\varepsilon_X)_*$ becomes an isomorphism after
inverting $t$, and the flatness of $\ZZ[1/t]$ yields
\eqref{eq:H2-away-from-type}. Finally, $\Lambda^2A_X$ is torsion-free and the kernel is $t^2$-torsion,
so the kernel is exactly $\Tor H_2^{\gr}(G)$.
\end{proof}

\section{Coefficient-contracted symplectic quandles}

We compute the second homotopy group of the rack space of a finite, possibly
nonconnected, symplectic quandle. We use Nosaka's calculations of associated
groups and low-dimensional group homology over finite fields
\cite[\S4.2]{NosakaAdjoint}; the degree-two quandle homology in the
nonconnected case and the vanishing of the third classifying homomorphism are
established here.

Let $q=\mathfrak p^d>10$, let $g\ge2$, and let
$V=\mathbb F_q^{2g}$ be equipped with a nondegenerate alternating form
$\omega$. Put
\[
 \Omega_q=\mathbb F_q^\times/
 \langle-1,(\mathbb F_q^\times)^2\rangle.
\]
Choose a representative $r_\lambda\in\mathbb F_q^\times$ for each
$\lambda\in\Omega_q$, let $V_\lambda^\times$ be a copy of
$V^\times=V\setminus\{0\}$, and define on
$X=\coprod_{\lambda\in\Omega_q}V_\lambda^\times$ the operation
\begin{equation}
 (x,\lambda)\lhd(y,\mu)
 =\bigl(x+r_\mu\omega(x,y)y,\lambda\bigr).
\label{eq:symp-operation}
\end{equation}
This is a coefficient-contracted variant of the nonzero symplectic quandle
\cite{NavasNelson}. Its isomorphism type is independent of the representatives.
Indeed, if $r'_\lambda=b_\lambda^2r_\lambda$, then
$(x,\lambda)\mapsto(b_\lambda^{-1}x,\lambda)$ is an isomorphism. If
$q\equiv3\pmod4$, then $\Omega_q$ is a singleton; after absorbing a square
factor, the remaining change is $r\mapsto-r$. Relative to a symplectic basis,
the map fixing one Lagrangian factor and negating the other has multiplier
$-1$, and hence gives the remaining isomorphism.

Put $G=\As(X)$, $\Gamma=\operatorname{Sp}_{2g}(\mathbb F_q)$, and
$\mathcal A=\ZZ[\Omega_q]$. Let
$\varepsilon_{\OO}\colon G\to\mathcal A$ be the orbit-degree map,
$\rho\colon G\to\Gamma$ the action homomorphism, and
$\varepsilon=\operatorname{aug}\circ\varepsilon_{\OO}$ the total degree,
where $\operatorname{aug}([\lambda])=1$.

\begin{proposition}
\label{prop:symp-associated-group}
The quandle $X$ has type $\mathfrak p$, and
\[
 \Inn(X)\cong\Gamma,
 \qquad
 \OO(X)\cong\Omega_q.
\]
Consequently, if $m=|\OO(X)|$, then
\begin{equation}
 m=
 \begin{cases}
 1,&q\text{ is even or }q\equiv3\pmod4,\\
 2,&q\equiv1\pmod4.
 \end{cases}
\label{eq:symp-orbit-number}
\end{equation}
Moreover, the canonical homomorphism
\begin{equation}
 (\varepsilon_{\OO},\rho)\colon
 G\xrightarrow{\cong}\mathcal A\times\Gamma
\label{eq:symp-associated-group}
\end{equation}
is an isomorphism, and
\begin{equation}
 H_2^{\gr}(G)\cong\Lambda^2\mathcal A,
 \qquad
 H_3^{\gr}(G)\cong\ZZ/(q^2-1).
\label{eq:symp-group-homology}
\end{equation}
\end{proposition}

\begin{proof}
Write $T_{a,y}(x)=x+a\omega(x,y)y$. A direct calculation gives
$T_{a,y}\in\Gamma$, $T_{a,y}(y)=y$, $T_{a,y}^{-1}=T_{-a,y}$, and
$gT_{a,y}g^{-1}=T_{a,g(y)}$ for $g\in\Gamma$; these identities give the
quandle axioms. Moreover, $T_{a,y}^n=T_{na,y}$, and nondegeneracy of
$\omega$ shows that every right translation has order $\mathfrak p$.

The diagonal action $g(x,\lambda)=(gx,\lambda)$ of $\Gamma$ on $X$ is
faithful and is by quandle automorphisms. For every
$c\in\mathbb F_q^\times$, there are $\mu\in\Omega_q$,
$a\in\mathbb F_q^\times$, and $\epsilon\in\{\pm1\}$ such that
$c=\epsilon r_\mu a^2$; hence $T_{c,y}$ is either $T_{r_\mu,ay}$ or its
inverse. Thus the right translations generate all symplectic transvections
and therefore $\Gamma$ \cite[Lemma~4.3]{NosakaAdjoint}. It follows that
$\Inn(X)\cong\Gamma$. Since the operation preserves the label of the left
entry and $\Gamma$ acts transitively on $V^\times$, the copies
$V_\lambda^\times$ are precisely the orbits, which also gives
\eqref{eq:symp-orbit-number}.

For $q>10$, one has
$H_1^{\gr}(\Gamma)=H_2^{\gr}(\Gamma)=0$ and
$H_3^{\gr}(\Gamma)\cong\ZZ/(q^2-1)$
\cite[Proposition~4.6 and its proof]{NosakaAdjoint}; see also
\cite{FiedorowiczPriddy,Friedlander}. By
\cite[Proposition~3.3]{NosakaAdjoint},
$K=\Ker\varepsilon_{\OO}$ is a perfect central extension of $\Gamma$.
Writing $C=\Ker(K\to\Gamma)$, the five-term exact sequence contains
$H_2^{\gr}(\Gamma)\to C\to H_1^{\gr}(K)$; both outer groups vanish, so
$C=0$. Thus $\rho|_K\colon K\to\Gamma$ is an isomorphism. Consequently,
$(\varepsilon_{\OO},\rho)$ is injective; it is surjective because one can
first prescribe the orbit degree and then correct the $\Gamma$-component by
an element of $K$. This proves \eqref{eq:symp-associated-group}, and
\eqref{eq:symp-group-homology} follows from the K\"unneth formula and
$m\le2$.
\end{proof}

Since the rack space $BX$ is simple
\cite[Proposition~5.2]{FRS}, its Hopf--Whitehead exact sequence takes the form
\cite[\S6.1]{NosakaHomotopical}
\begin{equation}
 H_3^\R(X)\xrightarrow{c_{3*}}H_3^{\gr}(G)
 \longrightarrow\pi_2(BX)
 \longrightarrow H_2^\R(X)\xrightarrow{c_{2*}}H_2^{\gr}(G)
 \longrightarrow0.
\label{eq:symp-whitehead}
\end{equation}

\begin{proposition}
\label{prop:symp-H2-kernels}
Let $\mathcal A_0=\Ker(\operatorname{aug}\colon\mathcal A\to\ZZ)$. Then
\begin{equation}
 H_2^\Q(X)
 \cong\bigoplus_{\lambda\in\Omega_q}\mathcal A_0
 \cong\ZZ^{m(m-1)}.
\label{eq:symp-H2Q}
\end{equation}
Moreover, noncanonically,
\[
 \Ker\bigl(c_{2*}\colon H_2^\R(X)\longrightarrow H_2^{\gr}(G)\bigr)
 \cong\ZZ^{m(m+1)/2}.
\]
\end{proposition}

\begin{proof}
Let $P<\Gamma$ be the stabilizer of a nonzero vector. By
\cite[Lemma~4.8]{NosakaAdjoint}, $P_{\ab}=0$. For
$a_\lambda\in V_\lambda^\times$, the isomorphism
\eqref{eq:symp-associated-group} gives
$\Stab_G(a_\lambda)\cap\Ker\varepsilon\cong\mathcal A_0\times P$.
Eisermann's Hurewicz isomorphism for nonconnected quandles therefore yields
\eqref{eq:symp-H2Q}
\cite[Definition~7.13 and Theorem~9.9]{Eisermann}.

The Litherland--Nelson splitting gives
$H_2^\R(X)\cong H_2^\Q(X)\oplus\ZZ[\OO(X)]\cong\ZZ^{m^2}$
\cite[Theorem~2.2]{LitherlandNelson}. By~\eqref{eq:symp-whitehead}, $c_{2*}$
is surjective, while
$H_2^{\gr}(G)\cong\ZZ^{m(m-1)/2}$. Hence its kernel is free of rank
$m^2-m(m-1)/2=m(m+1)/2$.
\end{proof}

\begin{remark}
\label{rem:comparison-nosaka-H2}
If $q\equiv1\pmod4$, then
$\Omega_q=\mathbb F_q^\times/(\mathbb F_q^\times)^2$, so $X$ is the
square-class model in \cite[\S4.2]{NosakaAdjoint}. In particular, $m=2$ and
$H_2^\Q(X)\cong\ZZ^2$. This corrects the vanishing assertion in
\cite[Proposition~4.7]{NosakaAdjoint}: in the nonconnected case,
Eisermann's formula uses the kernel of the total degree
$\varepsilon\colon\As(X)\to\ZZ$, so the factor $\mathcal A_0$ survives.
\end{remark}

\begin{theorem}
\label{thm:symp-pi2}
The homomorphism
$c_{3*}\colon H_3^\R(X)\to H_3^{\gr}(G)$ is zero. Moreover, there is a
noncanonical isomorphism
\[
 \pi_2(BX)\cong
 \ZZ^{m(m+1)/2}\oplus\ZZ/(q^2-1).
\]
\end{theorem}

\begin{proof}
Under~\eqref{eq:symp-associated-group}, the K\"unneth formula, $m\le2$, and
$H_1^{\gr}(\Gamma)=H_2^{\gr}(\Gamma)=0$ show that
$\rho_*\colon H_3^{\gr}(G)\to H_3^{\gr}(\Gamma)$ is an isomorphism. For
$x\in X$, one has
$\rho((e_x)^{\mathfrak p})=R_x^{\mathfrak p}=1$; hence
Proposition~\ref{prop:central-killing-composite} gives
$\mathfrak p\,\rho_*c_{3*}=0$. Thus
$\mathfrak p\,c_{3*}=0$. Since
$H_3^{\gr}(G)\cong\ZZ/(q^2-1)$ and
$\gcd(\mathfrak p,q^2-1)=1$, it follows that $c_{3*}=0$.

Sequence~\eqref{eq:symp-whitehead} now reduces to
\[
 0\longrightarrow H_3^{\gr}(G)
 \longrightarrow\pi_2(BX)
 \longrightarrow\Ker(c_{2*})
 \longrightarrow0.
\]
The right-hand term is free by Proposition~\ref{prop:symp-H2-kernels}; hence
the sequence splits noncanonically, and the stated formula follows from
\eqref{eq:symp-group-homology}.
\end{proof}

\begin{remark}
\label{rem:symp-pi2-nosaka}
For the standard connected one-copy quandle
$X_0=V^\times$, with operation $x\lhd y=x+\omega(x,y)y$, Nosaka proved
\[
 \pi_2(BX_0)\cong\ZZ\oplus\ZZ/(q^2-1)
\]
for $g\ge2$ and odd $q\notin\{3,5,7,9,27\}$
\cite[Equation~(18) and Theorem~7.4(I)]{NosakaHomotopical}.
When $q\equiv3\pmod4$ in this range, the present $X$ is isomorphic to
$X_0$, so Theorem~\ref{thm:symp-pi2} recovers his calculation. It also covers
the one-copy cases with even $q>10$ and with $q=27$. If
$q\equiv1\pmod4$, by contrast, the present $X$ is a different, two-orbit
quandle; in this case the theorem gives
\[
 \pi_2(BX)\cong\ZZ^3\oplus\ZZ/(q^2-1),
\]
and the two additional free summands reflect the nonconnected degree-two
homology in Proposition~\ref{prop:symp-H2-kernels}.
\end{remark}

\section{The second homology of the associated group of an Alexander quandle for which \texorpdfstring{$\mathrm{id}_M-T$}{idM-T} is invertible}

In this section, we compute the second homology of the associated group of a
connected Alexander quandle for which $\mathrm{id}_M-T$ is invertible. We
decompose Clauwens's presentation into a central extension and a semidirect
product, and combine two Lyndon--Hochschild--Serre spectral sequences to
derive the final formula.

Let $M$ be an abelian group and $T\in\operatorname{Aut}(M)$. Denote by
$X=Q_{M,T}$ the Alexander quandle with operation
\[
 x\lhd y=Tx+(\mathrm{id}_M-T)y.
\]
Throughout this section, assume that $\mathrm{id}_M-T$ is invertible, and put
\[
 A:=\Lambda^2M,
 \qquad
 L:=\Lambda^2T,
 \qquad
 f:=\mathrm{id}_A-L.
\]
Here, $H_2^\Q(X)\cong A/fA$ is shown by \cite[Corollary~4.3]{BIMNP}. We now compute
$H_2^{\gr}(\As(X))$.

By Clauwens's presentation \cite[Theorem~1]{ClauwensAlexander},
$G=\As(X)$ is isomorphic to
\[
        \ZZ\times M\times S,
        \qquad
        S=\Coker(\mathrm{id}_{M\otimes M}-\tau),
        \quad \tau(x\otimes y)=Ty\otimes x,
\]
where $\overline z$ denotes the class of $z\in M\otimes M$ in $S$, and the
multiplication is
\[
(k,x,\alpha)(m,y,\beta)
 =
 \bigl(k+m,T^m x+y,
 \alpha+\beta+\overline{T^m x\otimes y}\bigr).
\]
Let
$p\colon G\to\ZZ$, $(k,x,\alpha)\mapsto k$, be the projection onto the first
factor, and set $N:=\Ker p$. Then
\[
G=N\rtimes_\varphi\ZZ,
        \qquad
        \varphi(0,x,\alpha)=(0,T^{-1}x,\alpha),
\]
and
\begin{equation}
        1\longrightarrow S\longrightarrow N\longrightarrow M\longrightarrow1
\label{eq:alexander-central-extension}
\end{equation}
is a central extension. For an abelian group $B$ equipped with an action of
$\varphi$, write
\[
 B_\varphi:=B/(\mathrm{id}_B-\varphi)B,
 \qquad
 B^\varphi:=\Ker(\mathrm{id}_B-\varphi)
\]
for the $\varphi$-coinvariants and $\varphi$-invariants, respectively.

\begin{lemma}
\label{lem:alexander-extension}
The map
\[
 b\colon A\longrightarrow S,
 \qquad
 b(x\wedge y)=\overline{(\mathrm{id}_M-T)x\otimes y},
\]
induces an isomorphism $A/fA\cong S$. Moreover,
$S=[N,N]$, $H_1^{\gr}(N)\cong M$, and, up to sign, the transgression
$H_2^{\gr}(M)=A\to S$ of the extension
\eqref{eq:alexander-central-extension} is $b$.
\end{lemma}

\begin{proof}
The map
$(x,y)\mapsto\overline{(\mathrm{id}_M-T)x\otimes y}$ is biadditive. The
relation $\overline{x\otimes x}=\overline{Tx\otimes x}$ in $S$ implies
\[
 \overline{(\mathrm{id}_M-T)x\otimes x}=0,
\]
so $b$ is well defined on $A=\Lambda^2M$. Moreover, the endomorphism induced
by $\tau$ on $S$ is the identity, while $\tau^2=T\otimes T$. Hence
$b(Tx\wedge Ty)=b(x\wedge y)$, and therefore $bf=0$. Let
$\bar b\colon A/fA\to S$ be the induced map. Put
$R=(\mathrm{id}_M-T)^{-1}$. Then
\[
 \sigma(\overline{x\otimes y})=[Rx\wedge y]
\]
is well defined because
\[
 Rx\wedge y-R(Ty)\wedge x=f(Rx\wedge Ry).
\]
Furthermore,
$\bar b\sigma(\overline{x\otimes y})=\overline{x\otimes y}$ and
$\sigma\bar b(x\wedge y)=[x\wedge y]$. Thus $\sigma$ and $\bar b$ are
mutual inverses.

Using the convention $[g,h]=ghg^{-1}h^{-1}$, one computes
\[
 [(0,x,0),(0,y,0)]
 =(0,0,\overline{x\otimes y-y\otimes x})
 =(0,0,b(x\wedge y)).
\]
Since $b$ is surjective, $S=[N,N]$, and hence $N_{\ab}\cong M$. For a
central extension, the transgression agrees with the commutator map up to
sign \cite[Chapter~VII, \S6]{Brown}.
\end{proof}

\begin{theorem}
\label{thm:alexander-h2-associated}
Under the above assumptions,
\[
H_2^{\gr}(\As(Q_{M,T}))
        \cong fA/f^2A
        \cong
        \frac{\operatorname{Im}(\mathrm{id}_A-\Lambda^2T)}
        {\operatorname{Im}((\mathrm{id}_A-\Lambda^2T)^2)}.
\]
\end{theorem}

\begin{proof}
Consider the Lyndon--Hochschild--Serre spectral sequence of the central
extension~\eqref{eq:alexander-central-extension}. By
Lemma~\ref{lem:alexander-extension}, one has $d^2_{2,0}=\pm b$, and hence
$E^\infty_{2,0}=\Ker b=fA$. The homomorphism
$H_2^{\gr}(S)\to H_2^{\gr}(N)$ induced by inclusion is zero. Indeed,
$H_2^{\gr}(S)=\Lambda^2S$ is generated by the Pontryagin products
$\langle s_1,s_2\rangle_{\mathrm P}$. As observed immediately after
\eqref{eq:Pontryagin-class}, for fixed $s_1\in S\subset Z(N)$ the map
$s_2\mapsto\langle s_1,s_2\rangle_{\mathrm P}$ factors through $N_{\ab}$.
Since $S=[N,N]$, every such generator maps to zero. The image of this edge
homomorphism is the filtration subgroup
$F_0H_2^{\gr}(N)\cong E^\infty_{0,2}$; hence $E^\infty_{0,2}=0$, and the
filtration gives a
$\varphi$-equivariant exact sequence
\[
 M\otimes S\longrightarrow H_2^{\gr}(N)
 \longrightarrow fA\longrightarrow0.
\]

The automorphism $\varphi$ acts as $T^{-1}\otimes\mathrm{id}_S$ on
$M\otimes S$ and as $L^{-1}$ on $fA$. Since
$\mathrm{id}_M-T^{-1}$ is invertible, the endomorphism
$\mathrm{id}_{M\otimes S}-\varphi=(\mathrm{id}_M-T^{-1})\otimes\mathrm{id}_S$
is an automorphism. Hence $(M\otimes S)_\varphi=0$. Right exactness of
coinvariants therefore gives
\[
 H_2^{\gr}(N)_\varphi\cong(fA)_\varphi.
\]
On the other hand, the Lyndon--Hochschild--Serre spectral sequence for
$N\to G\to\ZZ$ yields
\[
 0\longrightarrow H_2^{\gr}(N)_\varphi
 \longrightarrow H_2^{\gr}(G)
 \longrightarrow H_1^{\gr}(N)^\varphi
 \longrightarrow0.
\]
Since $H_1^{\gr}(N)\cong M$ and
$M^{T^{-1}}=\Ker(\mathrm{id}_M-T^{-1})=0$, it follows that
\[
 H_2^{\gr}(G)\cong(fA)_\varphi
 =fA/(\mathrm{id}_A-L^{-1})fA=fA/f^2A.
\]
The last equality follows from
$\mathrm{id}_A-L^{-1}=-L^{-1}f$ and
$L^{-1}(f^2A)=f^2A$.
\end{proof}
We give an examples of Alexander quandles with $  H_2^{\gr}\bigl(\As(X_{\ell,r})\bigr) \neq 0. $
\begin{corollary}[A family with nontrivial second group homology]
\label{cor:alexander-nonvanishing-family}
Let $\ell\ge3$ be odd and let $r\ge3$. Put
$M=(\ZZ/\ell)^r$, choose a basis $e_1,\ldots,e_r$, and define
$T\in\operatorname{Aut}(M)$ by
\[
 Te_1=-e_1,\qquad
 Te_2=e_1-e_2,\qquad
 Te_j=-e_j\quad(3\le j\le r).
\]
Then $X_{\ell,r}:=Q_{M,T}$ is a finite connected Alexander quandle of
cardinality $\ell^r$, and
\[
 H_2^{\gr}\bigl(\As(X_{\ell,r})\bigr)
 \cong(\ZZ/\ell)^{r-2}, \qquad \qquad 
 H_2^\Q(X_{\ell,r})
 \cong(\ZZ/\ell)^{\binom r2-r+2}.
\]
In particular, the second homology of the associated group need not vanish,
even for finite connected Alexander quandles.
\end{corollary}

\begin{proof}
Let $D:=T+\mathrm{id}_M$. Then $D(e_2)=e_1$, while $D(e_i)=0$ for
$i\ne2$, and hence $D^2=0$. Thus $T=-\mathrm{id}_M+D$ is invertible.
Since $\ell$ is odd, multiplication by $2$ is invertible on $M$, and
$
 \mathrm{id}_M-T=2\mathrm{id}_M-D
$
is also invertible. Therefore $X_{\ell,r}$ is connected.

Let $A=\Lambda^2M$ and $f=\mathrm{id}_A-\Lambda^2T$. The elements
$e_i\wedge e_j$, $1\le i<j\le r$, form a basis of $A$ over $\ZZ/\ell$.
For $j\ge3$, one has
\[
 (\Lambda^2T)(e_2\wedge e_j)
 =(-e_2+e_1)\wedge(-e_j)
 =e_2\wedge e_j-e_1\wedge e_j,
\]
whereas every other basis element $e_i\wedge e_j$ is fixed by
$\Lambda^2T$. Consequently,
\[
 f(e_2\wedge e_j)=e_1\wedge e_j\quad(3\le j\le r),
 \qquad
 f(e_i\wedge e_j)=0
\]
for all remaining pairs $i<j$. Hence
\[
 fA=\bigoplus_{j : 3\leq j \leq r}(\ZZ/\ell)(e_1\wedge e_j),
 \qquad
 f^2A=0.
\]
Theorem~\ref{thm:alexander-h2-associated} now gives
$
 H_2^{\gr}\bigl(\As(X_{\ell,r})\bigr)
 \cong fA/f^2A
 \cong(\ZZ/\ell)^{r-2}.
$
Finally, the isomorphism $H_2^\Q(X_{\ell,r})\cong A/fA$ gives the stated
formula for the second quandle homology.
\end{proof}

\section*{Acknowledgments}
The author thanks Toshiyuki Akita and Kakeru Shikata for helpful discussions
and advice concerning this work.

\end{document}